\documentclass[11pt]{article}

\usepackage[dvips]{epsfig} 
\usepackage{amssymb,amsmath,amsthm,mathrsfs} 
\usepackage{graphicx}
\usepackage{lineno}
\usepackage{multirow}
\usepackage{longtable}

\usepackage[authoryear,sort&compress]{natbib}
\usepackage{url}
\usepackage{enumerate}
\usepackage{colordvi,color,pspicture}
\usepackage{psfrag}
\usepackage{pgf,tikz}
\usepackage{mathrsfs}
\usetikzlibrary{arrows}

\newcommand{\E}{\mathbf{E}}

\newtheorem{thm}{Theorem}[section]
\newtheorem{cor}[thm]{Corollary}

\newtheorem{prop}[thm]{Proposition}

\newtheorem{remark}[thm]{Remark}

\theoremstyle{definition}
\newtheorem{definition}[thm]{Definition}
\newtheorem{example}[thm]{Example}

\begin{document}
\title{A Classification of Isomorphism-Invariant Random Digraphs}
\author{
Selim Bahad{\i}r \& Elvan Ceyhan\\
Department of Mathematics, Ko\c{c} University,\\
Sar{\i}yer, 34450, Istanbul, Turkey.
}
\date{\today}

\maketitle
\pagenumbering{roman} \setcounter{page}{1}

\pagenumbering{arabic} \setcounter{page}{1}

\begin{abstract}
\noindent
We classify isomorphism-invariant random digraphs according to where randomness resides, namely,
arcs, vertices, and vertices and arcs together which in turn yield
arc random digraphs (ARD), vertex random digraphs (VRD) and vertex-arc random digraphs (VARD), respectively.
This digraph classification can be viewed as an extension of the classification of isomorphism-invariant random graphs.
We introduce randomness in the direction of the edges of a given graph and obtain direction random digraphs (DRD) as well.
We classify DRDs according to which component is random in addition to the direction
and study the relations of DRDs with VARDs, VRDs and ARDs.
We also consider random nearest neighbor digraphs and determine their membership with respect to these digraph families.
\end{abstract}


\section{Introduction}
\label{sec:intro}
A \emph{directed graph} (or simply \emph{digraph}) $D$ consists of a non-empty finite set
$V (D)$ of elements called \emph{vertices}
and a finite set $A(D)$ of ordered pairs of distinct vertices called \emph{arcs} (or \emph{directed edges}).
We call $V (D)$ the vertex set and $A(D)$ the arc set of $D$.
We will often denote $D$ as $D = (V,A)$.

For an arc $(u, v)$, the vertex $u$ is called the \emph{tail} and the vertex $v$ is called the \emph{head}.
The head and tail of an arc are called the \emph{end-vertices}.
The above definition of a digraph implies that we allow a digraph to have arcs with the same end-vertices
(for example, both $(u,v)$ and $(v,u)$ may be in $A$).
In this paper we only consider simple digraphs.
That is, we do not allow \emph{parallel} (also called \emph{multiple}) arcs, i.e.,
pairs of arcs with the same tail and the same head, or \emph{loops} (i.e., arcs whose heads and tails coincide).
When parallel arcs and loops are admissible we speak of \emph{directed pseudographs};
directed pseudographs without loops are \emph{directed multigraphs} (\cite{chartrand:1996}).
For more information about graphs and digraphs see, e.g., \cite{chartrand:1996}.

For a positive integer $n$,
let $[n]=\{1,2,\dots ,n\}$,
$\mathcal{D}_n$ denote the set of all digraphs with vertex set $[n]$ and
$2^{\mathcal{D}_n}$ denote the set of all subsets of $\mathcal{D}_n$.
A \emph{random digraph} is a probability space $(\mathcal{D}_n,2^{\mathcal{D}_n}, P)$,
and we write $\mathbf{D}= ( \mathcal{D}_n, P)$
where $P$ is a probability measure.
We call a random digraph as \emph{degenerate} if all the probability mass is on one digraph.
We can also think of $\mathbf{D}$ as
the outcome of an experiment of picking a digraph from $\mathcal{D}_n$ with distribution $P$.
For every $D\in \mathcal{D}_n$, we write $P(\{D\})$ as $P(D)$ for brevity in notation.
Also, for a measure space $(\Omega, \mathcal{F}, \mu )$,
$\mathcal{F}^n$ and $ \mu^n $ denote the usual product $\sigma$-algebra and product measure, respectively.
For the set of real numbers, we consider the Borel $\sigma$-algebra, and
throughout this paper we suppress the $\sigma$-algebra notation as long as
there is no necessity nor ambiguity.

\begin{example}
\label{ex:Dnm}
(Uniform Random Digraph Model)
For positive integers $n$ and $m$ with $n\geq 2$ and $0<m< n(n-1)$, $\mathbf{D}(n,m)$ is the random digraph such that
\begin{align*}
 P(D) =
  \begin{cases}
   \frac{1}{{n(n-1) \choose m}}, & \text{if } |A(D)|=m \\
   0, &  {\rm otherwise}
  \end{cases}
\end{align*}
for every $D\in \mathcal{D}_n$.
In other words, $\mathbf{D}(n,m)$ picks a digraph uniformly at random among the ones with vertex set $[n]$
and having exactly $m$ arcs.
Note that there are ${ n(n-1) \choose m}$ such digraphs, and $m$ is not chosen to be 0 or $n(n-1)$ to obtain a non-degenerate random digraph.
Also notice that $\mathbf{D}(n,m)$ is the digraph version of the Erd\H{o}s-R\'{e}nyi random graph $\mathbf{G}(n,m)$ (\cite{erdos:1959}).
For some asymptotic properties of uniform random digraphs see \cite{luczak:1990} and \cite{graham:2008}.
\end{example}

A digraph $D_1$ is \emph{isomorphic} to a digraph $D_2$ (or $D_1$ and $D_2$ are \emph{isomorphic}) if there is a bijection $f:V(D_1)\rightarrow V(D_2)$ such that
$(u,v)\in A(D_1)$ if and only if $(f(u),f(v))\in A(D_2)$.
\begin{definition}
(Isomorphism Invariance)
Let $\mathbf{D}=(\mathcal{D}_n,P)$ be a random digraph.
We say that $\mathbf{D}$ is \emph{isomorphism-invariant} if $P(D_1)=P(D_2)$
whenever $D_1$ and $D_2$ are isomorphic digraphs in $\mathcal{D}_n$.
\end{definition}

Throughout the article, we only consider non-degenerate isomorphism-invariant random digraphs.
We follow the isomorphism-invariant graph classification of \cite{beer:2011}
pointing out the similarities and differences until we introduce randomness in the direction.

In Section \ref{sec:2}, we introduce the arc random digraphs (ARDs), vertex random digraphs (VRDs)
and vertex-arc random digraphs (VARDs).
In Section \ref{sec:3}, for $n\geq 4$,
we prove that there is no random digraph which is both an ARD and a VRD,
and there exist VARDs which are neither ARDs nor VRDs.
Section \ref{sec:drd} introduces the direction random digraphs (DRDs), direction-edge random digraphs (DERDs),
direction-vertex random digraphs (DVRDs) and direction-vertex-edge random digraphs (DVERDs).
Section \ref{sec:derdvard} examines the relations of DERDs with ARDs and VARDs.
In particular, we show that ARDs are the only random digraphs which are both DERD and VARD for $n\geq 4$,
and any DERD with $n\leq 3$ is a VARD.
Section \ref{sec:rnnd} presents random nearest neighbor digraphs (RNNDs)
and determines where they fit in these classifications.
Discussion and conclusions are provided in Section \ref{sec:dis}.
A list of abbreviations used in the article is provided in Table \ref{tab:abbreviations}.

\section{Preliminaries}
\label{sec:prelim}
We first summarize isomorphism-invariant random graphs introduced by \cite{beer:2011}.
A \emph{graph} $G$ is a finite non-empty set $V(G)$ of elements called \emph{vertices} together with a set $E(G)$ of
unordered pairs of vertices of $G$ called \emph{edges}.
An edge $\{u,v\}$ is denoted by $uv$ for convenience in the text.
Let $\mathcal{G}_n$ denote the set of all graphs with $V(G)=[n]$ and
$2^{\mathcal{G}_n}$ be the set of all subsets of $\mathcal{G}_n$.
A \emph{random graph} is a probability space $(\mathcal{G}_n,2^{\mathcal{G}_n}, P)$,
and we write $\mathbf{G}= ( \mathcal{G}_n, P)$ where $P$ is a probability measure.
We write $P(G)$ instead of $P(\{G\})$ for brevity in notation

\begin{table}[t]
\centering
\begin{tabular}{llll}
ARD:  & Arc Random Digraph &(p. \pageref{def:ard})\\

DERD:  & Direction-Edge Random Digraph &(p. \pageref{def:derd})\\

DRD:  & Direction Random Digraph &(p. \pageref{def:drd})\\

DVERD:  & Direction-Vertex-Edge Random Digraph &(p. \pageref{def:dvrd})\\

DVRD:  & Direction-Vertex Random Digraph &(p. \pageref{def:dvrd})\\

ERG: & Edge Random Graph &(p. \pageref{def:erg})\\

GARD: & Generalized Arc Random Digraph &(p. \pageref{def:gard})\\

RNND: & Random Nearest Neighbor Digraph &(p. \pageref{def:rnnd})\\

VARD:  & Vertex-Arc Random Digraph &(p. \pageref{def:vard})\\

VERG: & Vertex-Edge Random Graph &(p. \pageref{def:verg})\\

VRD:  & Vertex Random Digraph &(p. \pageref{def:vrd})\\

VRG:  & Vertex Random Graph &(p. \pageref{def:vrg})\\
\end{tabular}
\caption{
\label{tab:abbreviations}
A list of abbreviations used in the article together with the page numbers where they are formally defined.}
\end{table}

The random graph model was first introduced by \cite{gilbert:1959} and \cite{erdos:1959}.
The model of Gilbert corresponds to edge random graph $\mathbf{G}(n,p_e)$ in \cite{beer:2011}
in which each edge is inserted, independent of others, with probability $p_e$.
The model introduced by Erd\H{o}s and R\'{e}nyi is the uniform random graph $\mathbf{G}(n,m)$ which
picks a graph with vertex set $[n]$ uniformly at random among the ones with exactly $m$ edges.
However, in the literature,
both of these models are usually called Erd\H{o}s-R\'{e}nyi model as they developed the theory.

A graph $G_1$ is \emph{isomorphic} to a graph $G_2$ (or $G_1$ and $G_2$ are \emph{isomorphic}) if there exists a bijection $f:V(G_1)\rightarrow V(G_2)$ such that
$uv \in E(G_1)$ if and only if $f(u)f(v)\in E(G_2)$.
We say that the random graph $\mathbf{G}=(\mathcal{G}_n,P)$ is \emph{isomorphism-invariant} if $P(G_1)=P(G_2)$ whenever $G_1$ is isomorphic to $G_2$.

\begin{definition}\label{def:erg}
An \emph{edge random graph} (ERG) is a random graph $\mathbf{G}(n,p_e)=(\mathcal{G}_n, P)$ where $p_e\in [0,1]$ and
\begin{align*}
P(G)=p_e^{|E(G)|} (1-p_e)^{{n\choose 2}-|E(G)|} \text{ for every }  G\in \mathcal{G}_n .
\end{align*}
\end{definition}

Let $\Omega$ be a set, $\mathbf{x}=(x_1,\dots ,x_n) \in \Omega^n$ and
$\phi: \Omega \times \Omega \rightarrow \{0,1\}$ be a symmetric function.
Then the $(\mathbf{x},\phi )$\emph{-graph}, denoted $G(\mathbf{x},\phi)$, is defined to be the graph, $G$,
with vertex set $[n]$ such that for every $i,j\in [n]$ with $i\neq j$
we have $ij\in E(G)$ if and only if $ \phi (x_i,x_j)=1$.

\begin{definition}\label{def:vrg}
Let $(\Omega, \mathcal{F}, \mu )$ be a probability space and
$\phi : \Omega \times \Omega \rightarrow \{0,1\}$  be a symmetric measurable function.
The \emph{vertex random graph} (VRG), $\mathbf{G}(n, \Omega , \mu , \phi )$,
is the random graph $(\mathcal{G}_n, P)$ satisfying
\begin{align*}
P(G)= \int \mathbf{1}_{ \{ G(\mathbf{x},\phi )=G \} } d(\mu \mathbf{x}) \text{ for every } G\in \mathcal{G}_n,
\end{align*}
where $d(\mu \mathbf{x})$ is short-hand for the product integrator $d(\mu ^n (\mathbf{x}))=d(\mu x_1)\cdots d(\mu x_n)$.
\end{definition}
Notice that in a VRG the randomness lies in the structure attached to the vertices, and
once these random structures have been assigned to the vertices, all the edges are uniquely determined.

\begin{definition}\label{def:verg}
Let $(\Omega, \mathcal{F}, \mu )$ be a probability space and
$\phi :\Omega \times \Omega \rightarrow [0,1]$ be a symmetric measurable function.
The \emph{vertex-edge random graph} (VERG), $\mathbf{G}(n, \Omega, \mu ,\phi )$, is the random graph $(\mathcal{G}_n, P)$ with
\begin{align*}
P(G)= \int P_{\mathbf{x}}(G) d(\mu \mathbf{x}), \text{ for every } G\in \mathcal{G}_n,
\end{align*}
where for given $\mathbf{x}=(x_1,\dots ,x_n)$ and $G$
\begin{align*}
P_{\mathbf{x}}(G)=\prod_{ij\in E(G)} \phi (x_i,x_j) \times \prod_{ij\notin E(G)} (1- \phi (x_i,x_j)).
\end{align*}
\end{definition}
In words, a VERG is generated as follows: a random sample of size $n$ is drawn with distribution $\mu$ from $\Omega$,
say $\mathbf{X}=(X_1,\dots , X_n)$.
Then conditional on $\mathbf{X}$, independently for each pair of distinct vertices $i$ and $j$,
the edge $ij$ is inserted with probability $\phi (X_i,X_j)$.

Observe that the same notation $\mathbf{D}(n, \Omega,\mu, \phi)$ is used for both VRGs and VERGs.
However, this causes no confusion, since $\phi$ takes values in $\{0,1\}$ for VRGs and in $[0,1]$ for VERGs.
In other words, VRGs form a special case of VERGs with $\phi$ taking values only in $\{0,1\}$.
Therefore, every VRG is a VERG.
In addition, it is easy to see that letting $\phi$ to be identically equal to $p$ gives that every ERG is a VERG.

Let $\mathbf{G_1}=(\mathcal{G}_n,P_1)$ and $\mathbf{G_2}=(\mathcal{G}_n,P_2)$ be random graphs.
The \emph{total variation distance} between $\mathbf{G_1}$ and $\mathbf{G_2}$ is defined to be
\begin{align*}
d_{ \text{TV} } (\mathbf{G_1},\mathbf{G_2})= \frac{1}{2} \sum_{G\in \mathcal{G}_n} |P_1(G)-P_2(G)|.
\end{align*}
Similarly, for any two random digraphs $\mathbf{D_1}=(\mathcal{D}_n,P_1)$ and $\mathbf{D_2}=(\mathcal{D}_n,P_2)$,
the \emph{total variation distance} between $\mathbf{D_1}$ and $\mathbf{D_2}$ is defined to be
\begin{align*}
d_{ \text{TV} } (\mathbf{D_1},\mathbf{D_2})= \frac{1}{2} \sum_{D\in \mathcal{D}_n} |P_1(D)-P_2(D)|.
\end{align*}

\section{ARDs, VRDs and VARDs}
\label{sec:2}
\subsection{Arc random digraphs}
\label{subsec:ard}
One of the most commonly studied random digraphs is the binomial (or Bernoulli) random digraph model, $\mathbf{D}(n,p_a)$, in which each of the $n(n-1)$ possible arcs is included independently with probability $p_a$.
Such random digraphs give rise to arc random digraphs.
\begin{definition}\label{def:ard}
An \emph{arc random digraph} (ARD) is a random digraph $\mathbf{D}(n,p_a)=(\mathcal{D}_n, P)$ where $0<p_a<1$ and
\begin{align*}
P(D)=p_a^{|A(D)|} (1-p_a)^{n(n-1)-|A(D)|} \  \text{ for every }  D\in \mathcal{D}_n .
\end{align*}
\end{definition}
Notice that ARDs are the digraph counterparts of random graphs $\mathbf{G}(n,p_e)$ due to \cite{gilbert:1959}.
For some asymptotic properties of $\mathbf{D}(n,p_a)$ see \cite{Karp:1990}, \cite{luczak:2009}, and \cite{Krivelevich:2013}.

\begin{definition}\label{def:gard}
Let $\mathrm{p_a}:[n]\times [n]\rightarrow [0,1]$ be a function (that is not necessarily symmetric in its arguments).
The \emph{generalized arc random digraph} (GARD), $\mathbf{D}(n,\mathrm{p_a})$, is the random digraph $(\mathcal{D}_n, P)$ with
\begin{align*}
P(D)=\prod_{(i,j)\in A(D)} \mathrm{p_a}(i,j) \times \prod_{(i,j)\notin A(D)} (1-\mathrm{p_a}(i,j)) \ \text{ for every } D\in \mathcal{D}_n.
\end{align*}
\end{definition}
In other words, in a GARD each arc appears independently of others and
the arc $(i,j)$ occurs with probability $\mathrm{p_a}(i,j)$.
Note that an ARD is special case of a GARD with a constant $\mathrm{p_a}$, i.e., $\mathrm{p_a}(i,j)=p_a$ for all $i,j$.
As the classical random digraph model $\mathbf{D}(n,p_a)$ may not fit real life networks,
inhomogeneous models like GARDs are of interest for such scenarios (see, e.g., \cite{bloznelis:2012}).

Clearly, any ARD is isomorphism-invariant.
The following proposition implies that a GARD is isomorphism-invariant if and only if it is an ARD.
\begin{prop}
\label{prop:gard}
Let $\mathbf{D}$ be an isomorphism-invariant GARD.
Then $\mathbf{D}=\mathbf{D}(n,p_a)$ for some $p_a$, i.e., $\mathbf{D}$ is an ARD.
\end{prop}
\begin{proof}
We show that $\mathrm{p_a}(i,j)=\mathrm{p_a}(k,l)$ for any two ordered pairs $(i,j)$ and $(k,l)$.
First note that
\begin{align}
\label{eq:gard1}
\mathrm{p_a}(i,j)=P((i,j)\in A(\mathbf{D}))=\sum_{(i,j)\in A(D)} P(D).
\end{align}
Fix a permutation on $[n]$ which maps $i$ to $k$ and $j$ to $l$.
Observe that this permutation induces a one-to-one correspondence between the sets $\{D\in\mathcal{D}_n: (i,j)\in A(D)\}$ and $\{D'\in\mathcal{D}_n: (k,l)\in A(D')\}$ such that matched digraphs are isomorphic.
As $\mathbf{D}$ is isomorphism-invariant, this correspondence implies
\begin{align}
\label{eq:gard2}
\sum_{(i,j)\in A(D)} P(D)=\sum_{(k,l)\in A(D')} P(D').
\end{align}
Hence, the result follows by \eqref{eq:gard1} and \eqref{eq:gard2}.
\end{proof}

\subsection{Vertex random digraphs}
\label{subsec:vrd}

Let $\Omega$ be a set, $\mathbf{x}=(x_1,\dots ,x_n) \in \Omega^n$ and
$\phi: \Omega \times \Omega \rightarrow \{0,1\}$ be a function.
Then the $(\mathbf{x},\phi )$\emph{-digraph}, denoted $D(\mathbf{x},\phi)$, is defined to be the digraph, $D$,
with vertex set $[n]$ such that for all $i,j\in [n]$ with $i\neq j$ we have
\begin{align*}
(i,j)\in A(D) \text{ if and only if } \phi (x_i,x_j)=1.
\end{align*}

Clearly, every digraph $D$ with $V(D)=[n]$ is an $(\mathbf{x},\phi)$-digraph for some choice of $\Omega,\mathbf{x}$ and $\phi$.
More specifically, choose $\mathbf{x}$ to be the identity function on $\Omega=[n]$ and
define $\phi (i,j)=\mathbf{1}_{ \{ (i,j)\in A(D) \} }$ where $\mathbf{1}_{\{\cdot \}}$ is the indicator function.

\begin{definition}\label{def:vrd}
Let $(\Omega, \mathcal{F}, \mu )$ be a probability space and
$\phi : \Omega \times \Omega \rightarrow \{0,1\}$  be a measurable function.
The \emph{vertex random digraph} (VRD), $\mathbf{D}(n, \Omega , \mu , \phi )$, is the random digraph $(\mathcal{D}_n, P)$ with
\begin{align*}
P(D)= \int \mathbf{1}_{ \{ D(\mathbf{x},\phi )=D \} } d(\mu \mathbf{x}) \text{ for every } D\in \mathcal{D}_n.
\end{align*}
\end{definition}
Note that in a VRD the randomness resides in the structure attached to the vertices, as in VRGs, and
when these random structures are assigned to the vertices, all the arcs are uniquely determined.
\begin{example}
Proximity Catch Digraphs (PCDs)( \cite{ceyhan:2011}):
Let $(\Omega, \mathcal{F}, \mu)$ be a probability space.
The \emph{proximity map} $N(\cdot )$ is a function from $\Omega$ to $\mathcal{F}$.
The \emph{proximity region} associated with $x\in \Omega$, denoted $N(x)$, is the image of $x\in \Omega$ under $N(\cdot )$.
The points in $N(x)$ are thought of as being ``closer'' to $x\in \Omega$ than the points in $\Omega\backslash N(x)$.
For a given $\mathbf{x}=(x_1,x_2,\dots ,x_n)$ the \emph{proximity catch digraph} is the digraph with the vertex set $V=[n]$
and the arc set $A=\{(i,j): x_j\in N(x_i)\}$.
In other words, we insert the arc $(i,j)$ if and only if $x_j$ is in the proximity region of $x_i$.
Note that for a given $N(\cdot)$, a random PCD is a VRD, $\mathbf{D}(n,\Omega,\mu,\phi)$,
with $\phi(x_i,x_j)=\mathbf{1}_{ \{x_j\in N(x_i)\} }$.
For instance, one can take $\Omega=\mathbb{R}$, $N(x)=[x,\infty)$ and $\phi(x,y)=\mathbf{1}_{ \{x\leq y \}}$.
\end{example}

\begin{example}
Random Intersection Digraphs (\cite{bloznelis:2010}):
Let $n$ and $m$ be positive integers,
and $\mu$ be a distribution on $2^{[m]}\times 2^{[m]}$ (ordered pairs of subsets of $[m]$).
Given two collections of subsets $S_1,\dots , S_n$ and $T_1,\dots , T_n$ of the set $[m]$,
define the intersection digraph with vertex set $[n]$ such that
the arc $(i,j)$ is present in the digraph whenever $S_i \cap T_j$ is nonempty for $i\neq j$.
$\mathbf{D}(n,m,\mu)$ is the random intersection digraph generated by independent and identically distributed pairs of random subsets $(S_i,T_i)$ under $\mu$, $1\leq i\leq n$.
Note that $\mathbf{D}(n,m,\mu)$ is a VRD with $\Omega =2^{[m]}\times 2^{[m]}$ and
$\phi((S,T),(S',T'))=\mathbf{1}_{ \{ S \cap T' \neq \emptyset \}}$.
\end{example}

By letting $\Omega=[0,1]$, $\mu$ be the uniform distribution over $[0,1]$ and $\phi(x,y)=\mathbf{1}_{\{x\leq p_a\}}$,
we see that every $\mathbf{D}(2,p_a)$ is a VRD.

Recall that in a VRD, $\mathbf{D}(n,\Omega,\mu, \phi)$, $\phi$ is not required to be symmetric.
However, if $\phi$ is a symmetric function, whenever we see the arc $(i,j)$ in $A(\mathbf{D})$,
we see the arc $(j,i)$ as well.
In this case, for every $D\in \mathcal{D}_n$ in which there exists $(i,j)\in A(D)$ with $(j,i)\notin A(D)$, we have $P(D)=0$.
On the other hand, in an ARD, $\mathbf{D}(n,p_a)$, we have $P(D)>0$ for every $D\in \mathcal{D}_n$.
Therefore, whenever $\phi$ is symmetric and nonconstant $\mu^2$-a.s.,
$\mathbf{D}(n,\Omega,\mu, \phi)$ is not an ARD.
For instance, one can take $\Omega=\mathbb{R}^d$, $\mu$ to be an a.e. continuous distribution and $\phi(x,y)=\mathbf{1}_{ \{||x-y||_d \leq r \} }$,
where $||\cdot ||_d$ is the usual Euclidean norm in $\mathbb{R}^d$ and $r$ is a fixed positive real number.
Notice that these random digraphs are random PCDs in which $N(x)$ is the closed ball with radius $r$ and center $x$.
If we consider symmetric arcs as one edge only,
these type of random digraphs reduce to what is called \emph{random geometric graphs}.
For more information about random geometric graphs see \cite{penrose:2003}.

\subsection{Vertex-arc random digraphs}
\label{subsec:vard}
We now generalize the random digraphs introduced in the previous two subsections by combining the structures where the randomness lies.
\begin{definition}\label{def:vard}
Let $(\Omega, \mathcal{F}, \mu)$ be a probability space and
$\phi :\Omega \times \Omega \rightarrow [0,1]$ be a measurable function.
The \emph{vertex-arc random digraph} (VARD), $\mathbf{D}(n, \Omega, \mu ,\phi )$, is the random digraph $(\mathcal{D}_n, P)$ with
\begin{align*}
P(D)= \int P_{\mathbf{x}}(D) d(\mu \mathbf{x}), \text{ for every } D\in \mathcal{D}_n,
\end{align*}
where for given $\mathbf{x}=(x_1,\dots ,x_n)$ and $D=(V,A)$
\begin{align*}
P_{\mathbf{x}}(D)=\prod_{(i,j)\in A} \phi (x_i,x_j) \times \prod_{(i,j)\notin A} (1- \phi (x_i,x_j)).
\end{align*}
\end{definition}
The construction of a VARD is almost same with VERGs.
A random sample of size $n$ is drawn with distribution $\mu$ from $\Omega$,
say $\mathbf{X}=(X_1,\dots , X_n)$, and
then conditional on $\mathbf{X}$, independently for each pair of distinct vertices $i$ and $j$,
the arc $(i,j)$ is inserted with probability $\phi (X_i,X_j)$.

Note that we use the same notation $\mathbf{D}(n, \Omega,\mu, \phi)$ for both VRDs and VARDs.
But, since $\phi$ takes values only 0 or 1 for VRDs and in $[0,1]$ for VARDs,
this causes no confusion.
Particularly, VRDs form a special case of VARDs with $\phi$ taking values only in $\{0,1\}$.
Therefore, every VRD is a VARD.
Moreover, it is easy to verify that letting $\phi$ to be identically equal to $p_a$ gives that every ARD is a VARD.

\begin{prop}
\label{prop:isoinv}
Every VARD is isomorphism-invariant.
\end{prop}
\begin{proof}
Let $\mathbf{D}(n, \Omega, \mu ,\phi )$ be a VARD and $D,D'\in \mathcal{D}_n$ be isomorphic digraphs.
Then there exists a permutation $\sigma$  on $[n]$ such that
\begin{align*}
(i,j)\in A(D) \ \Leftrightarrow \ (\sigma(i),\sigma(j))\in A(D').
\end{align*}
Let $\sigma^{-1}$ be the inverse of $\sigma$ and $\mathbf{y}=(y_1,\dots ,y_n)$ such that
$y_i=x_{\sigma^{-1}(i)}$ for all $1\leq i \leq n$, i.e., $x_i=y_{\sigma(i)}$ for all $1\leq i \leq n$.
Then note that
\begin{align}
P_{\mathbf{x}}(D)&=\prod_{(i,j)\in A(D)} \phi (x_i,x_j) \times \prod_{(i,j)\notin A(D)} (1- \phi (x_i,x_j)) \nonumber\\
&=\prod_{(i,j)\in A(D)} \phi (y_{\sigma(i)},y_{\sigma(j)}) \times \prod_{(i,j)\notin A(D)} (1- \phi (y_{\sigma(i)},y_{\sigma(j)})) \nonumber \\
&=\prod_{(\sigma(i),\sigma(j))\in A(D')} \phi (y_{\sigma(i)},y_{\sigma(j)}) \times \prod_{(\sigma(i),\sigma(j))\notin A(D')} (1- \phi (y_{\sigma(i)},y_{\sigma(j)})) \nonumber \\
&=\prod_{(i,j)\in A(D')} \phi (y_i,y_j) \times \prod_{(i,j)\notin A(D')} (1- \phi (y_i,y_j)) \nonumber \\
&=P_{\mathbf{y}}(D'). \label{eq:iso1}
\end{align}
As $\mathbf{y}$ is a permutation of $\mathbf{x}$, Fubini's theorem and \eqref{eq:iso1} imply that
\begin{align}
\label{eq:iso2}
P(D)=\int P_{\mathbf{x}}(D) \mu(d\mathbf{x})=\int P_{\mathbf{y}}(D') \mu(d\mathbf{y}).
\end{align}
Furthermore, the change of variables that maps $y_i$ to $x_i$ in the integrant above results
\begin{align}
\label{eq:iso3}
\int P_{\mathbf{y}}(D') \mu(d\mathbf{y})=\int P_{\mathbf{x}}(D') \mu(d\mathbf{x})=P(D'),
\end{align}
since the mapping is a permutation and the Jacobian of a permutation matrix is $\pm 1$.
Thus, the results in \eqref{eq:iso2} and \eqref{eq:iso3} together imply that $P(D)=P(D')$,
and so the desired result follows.
\end{proof}

As a corollary, we easily see that any VRD is isomorphism-invariant since every VRD is a VARD.

\section{Inclusion/exclusion relations between ARDs, VRDs and VARDs}
\label{sec:3}
In the previous section we have shown that every ARD is a VARD and so is every VRD,
and every VARD is isomorphism-invariant.
In this section we prove that for $n\geq 4$ there exists no random digraph which is both ARD and VRD,
and the union of the classes ARDs and VRDs is not the entire class of VARDs.

The following theorem implies that the families ARDs and VRDs are disjoint for $n\geq 4$.
\begin{thm}
\label{thm:1}
If an ARD, $\mathbf{D}(n,p_a)$, with $n\geq 4$ is represented as a VARD, $\mathbf{D}(n,\Omega, \mu, \phi)$,
then $\phi (x,y)=p_a$ $\mu^2$-a.s.
\end{thm}

\begin{proof}
Suppose that an ARD, $\mathbf{D}(n,p_a)$, with $n\geq 4$ is represented as a VARD, $\mathbf{D}(n,\Omega, \mu, \phi)$.
For the proof of the theorem,
we borrow some tools from functional analysis which are presented in the proof of the Theorem 4.2. in \cite{beer:2011}.
Let $h: \Omega\times \Omega \rightarrow [0,1]$ be a symmetric measurable function
and $T$ be the integral operator with kernel $h$ on the space $L^2(\Omega,\mu)$ of $\mu$-square-integrable functions on $\Omega$:
\begin{align*}
(Tg)(x)=\int h(x,y)g(y)d(\mu y).
\end{align*}
Since $h$ is bounded and $\mu$ is a finite measure, the kernel $h$ is in $L^2(\mu \times \mu )$.
Integral operators with such kernels are Hilbert-Schmidt operators and are thus compact operators.
Moreover, as $h$ is symmetric, the integral operator $T$ is self-adjoint,
which implies that $L^2(\Omega,\mu)$ has an orthonormal basis $(\psi_i)_{i\geq 1}$ of eigenfunctions for $T$
such that $T\psi _i=\lambda_i \psi_i$ for not necessarily distinct real eigenvalues $\lambda_i$
with $\lambda_i\rightarrow 0$ as $i\rightarrow \infty$ (see Chapter VI in \cite{reedsimon:1980}).
We may assume that $\lambda_1$ is the largest eigenvalue.
Then we have
\begin{align*}
h(x,y)=\sum_{i\geq 1} \lambda_i \psi_i(x)\psi_i(y) ~~~ \mu^2\text{-a.s.}
\end{align*}
with the sum converging in $L^2$.
As $\psi_i$'s are orthonormal, it follows that
\begin{align}
\label{eq:3.1}
&\E(h(X_1,X_2)h(X_2,X_3)h(X_3,X_4)h(X_4,X_1)) \nonumber\\
&=\int \int \int \int h(x_1,x_2)h(x_2,x_3)h(x_3,x_4)h(x_4,x_1)d(\mu x_1)d(\mu x_2)d(\mu x_3)d(\mu x_4) \nonumber \\
&=\sum_{i\geq 1} \lambda_i^4.
\end{align}
Now let $E_1$ be the event that both $(1,2)$ and $(2,1)$ are in $A(\mathbf{D})$.
As $\mathbf{D}$ is an ARD $\mathbf{D}(n,p_a)$, it is easy to see that $P(E_1)=p_a^2$.
On the other hand, since $\mathbf{D}$ is represented as a VARD, $\mathbf{D}(n,\Omega, \mu, \phi)$, we have
\begin{align*}
P(E_1)=\E(\phi(X_1,X_2)\phi(X_2,X_1)).
\end{align*}
Thus, letting $h(x,y)=\phi(x,y)\phi(y,x)$ gives $p_a^2=\E(h(X_1,X_2))$.
As
\begin{align*}
\E(h(X_1,X_2))=\int \int h(x,y)d(\mu x)d(\mu y)=\langle T\mathbf{1}, \mathbf{1} \rangle\leq \lambda_1,
\end{align*}
we get $p_a^2\leq \lambda_1$, where $\mathbf{1}$ is the function with constant value 1.

Let $E_2$ be the event that
$(1,2),(2,1),(2,3),(3,2),(3,4),(4,3),(4,1),(1,4)\in A(\mathbf{D})$.
Since $\mathbf{D}$ is an ARD, $\mathbf{D}(n,p_a)$, it is easy to see that
\begin{align}
\label{eq:3.3}
P(E_2)=p_a^8.
\end{align}
By the representation of $\mathbf{D}$ as a VARD, $\mathbf{D}(n,\Omega, \mu, \phi)$, we also have
\begin{align}
\label{eq:3.4}
P(E_2)=\E(h(X_1,X_2)h(X_2,X_3)h(X_3,X_4)h(X_4,X_1)).
\end{align}
Now combining the results in \eqref{eq:3.1}, \eqref{eq:3.3} and \eqref{eq:3.4} gives
\begin{align}
\label{eq:3.5}
p_a^8=\sum_{i\geq 1}\lambda_i^4.
\end{align}
Since $p_a^2\leq \lambda_1$, we have $p_a^8\leq \lambda_1^4$ and thus, by \eqref{eq:3.5} we obtain that $\lambda_1=p_a^2$
and $\lambda_i=0$ for every $i\geq 2$, that is $h(x,y)=p_a^2\psi_1(x)\psi_1(y)$.
But then we have
\begin{align*}
p_a^2\int \psi_1^2(x)d(\mu x)=p_a^2=\E(h(X_1,X_2))
=p_a^2\int \int \psi_1(x)\psi_1(y) d(\mu x)d(\mu y)=p_a^2 \left ( \int \psi_1(x)d(\mu x) \right )^2,
\end{align*}
which implies that
\begin{align}
\label{eq:3.6}
\int \psi_1^2(x)d(\mu x)= \left ( \int \psi_1(x)d(\mu x) \right )^2,
\end{align}
since $p\neq 0$.
As the equality in equation \eqref{eq:3.6} is the equality in the Cauchy-Schwarz inequality for $\psi_1$ and $\mathbf{1}$,
we see that $\psi_1$ is constant $\mu$-a.s.
Since $\int \psi_1^2(x)d(\mu x)=1$, we get $\psi_1=1$ $\mu$-a.s. or $\psi_1=-1$ $\mu$-a.s.,
and therefore $h(x,y)=p_a^2$ $\mu^2$-a.s., that is
\begin{align}
\label{eq:3.7}
\phi(x,y)\phi(y,x)=p_a^2~~~ \mu^2 \text{-a.s.}
\end{align}
Next, let $E_3$ be the event that neither of the arcs $(1,2)$ and $(2,1)$ is in $A(\mathbf{D})$,
and $E_4$ be the event that none of the arcs $(1,2),(2,1),(2,3),(3,2),(3,4),(4,3),(4,1),(1,4)$ is in $A(\mathbf{D})$.
Choosing $h(x,y)$ to be $(1-\phi(x,y))(1-\phi(y,x))$ allows us to follow the same arguments above for $E_1$ and $E_2$
replaced with $E_3$ and $E_4$, respectively, and with $1-p_a$ taking place of $p_a$.
Therefore, we obtain that
\begin{align}
\label{eq:3.8}
(1-\phi(x,y))(1-\phi(y,x))=(1-p_a)^2 ~~~ \mu^2 \text{-a.s.}
\end{align}
Finally, the equations in \eqref{eq:3.7} and \eqref{eq:3.8} give the desired result.
\end{proof}

\begin{remark}
The function $h$ in the proof of Theorem \ref{thm:1} is taken to be symmetric.
Otherwise, the operator $T$ needs not to be self-adjoint and
hence the succeeding arguments in the proof are not true.
So, $h$ being symmetric is a crucial condition for the proof.
If $\phi$ is given to be symmetric,
one can take $h=\phi$ and obtain $\phi=p_a$ $\mu^2$-a.s. (as in the proof of Theorem 4.2. in \cite{beer:2011}).
However, in our case, $\phi$ is not supposed to be symmetric.
Note that we tackle this hurdle by taking $h(x,y)$ to be $\phi(x,y)\phi(y,x)$ and
$(1-\phi(x,y))(1-\phi(y,x))$, respectively, and obtain the desired result.
\end{remark}

As any VRD is a VARD with $\phi$ taking values in $\{0,1\}$, by Theorem \ref{thm:1} we have the following corollary.
\begin{cor}
\label{cor:ardnotvrd}
Any ARD, $\mathbf{D}(n,p_a)$, with $n\geq 4$ is not a VRD.
\end{cor}

We next show that union of ARD and VRD families do not constitute the entire class of VARDs when $n\geq 4$.
\begin{thm}
\label{thm:onlyvard}
There exist VARDs with $n\geq 4$ which are neither a VRD nor an ARD.
\end{thm}
\begin{proof}
Let $0<a<b<1$ be real numbers.
Consider a VARD, $\mathbf{D}(n, \Omega,\mu, \psi)$, with $n\geq 4$ such that
$\phi(x,y)\in \{a,b\}$, and $\psi(x,y)\neq \psi(y,x)$ for any $x\neq y$.
Equivalently, we have
\begin{align}
\label{eq:psiab}
\psi(x,y)\psi(y,x)=ab \text{ and } (1-\psi(x,y)) (1-\psi(y,x))=(1-a)(1-b) \text{ for } x\neq y.
\end{align}
For example, one can take $\Omega=\mathbb{R}$, $\mu$ to be a continuous distribution and $\psi(x,y)=a\mathbf{1}_{ \{x\leq y\} } + b\mathbf{1}_{ \{y< x\} }$.

Now suppose that it has another VARD representation $\mathbf{D}(n, \Omega', \nu,\phi)$.
We claim that $\phi$ satisfies the same properties of $\psi$ given in \eqref{eq:psiab} $\nu^2$-a.s.
Recall that in the proof of Theorem \ref{thm:1}, the properties of an ARD, $\mathbf{D}(n,p_a)$, that we used are
\begin{align}
\label{eq:ardprops}
P(E_2)=p_a^8=(p_a^2)^4=(P(E_1))^4 \text{ and }  P(E_4)=(1-p_a)^8=((1-p_a)^2)^4=(P(E_3))^4.
\end{align}
Notice that the equations in \eqref{eq:ardprops} hold for $\mathbf{D}(n, \Omega,\mu, \psi)$ when $p_a^2$ and $(1-p_a)^2$ are replaced with $ab$ and $(1-a)(1-b)$, respectively.
That is, for $\mathbf{D}(n, \Omega,\mu, \psi)$ we have
\begin{align*}
P(E_2)=(ab)^4=(P(E_1))^4 \text{ and }  P(E_4)=((1-a)(1-b))^4=(P(E_3))^4.
\end{align*}
Therefore, following the same arguments in the proof of Theorem \ref{thm:1} we obtain
\begin{align*}
\phi(x,y)\phi(y,x)=ab  \text{ and } (1-\phi(x,y))(1-\phi(y,x))=(1-a)(1-b)~~~ \nu^2 \text{-a.s.}
\end{align*}
and hence the claim follows.
Then, by the choice of $a,b$ and $\psi$, we see that $\mathbf{D}(n, \Omega,\mu, \psi)$ has neither an ARD nor a VRD representation.
\end{proof}

Recall that Corollary \ref{cor:ardnotvrd} and Theorem \ref{thm:onlyvard} imply that for $n\geq 4$
$\text{ARD}\cap \text{VRD}=\emptyset$ and $\text{VARD}\backslash (\text{ARD} \cup \text{VRD} ) \neq \emptyset$, respectively, in Figure \ref{fig:vards}.

\begin{figure}\center
\scalebox{.8}{
\begin{tikzpicture}[line cap=round,line join=round,>=triangle 45,x=1.5cm,y=1.5cm]
\clip(0.8537834428690544,0.8212498658964582) rectangle (9.162062220187407,6.526727615465726);
\draw (1.,6.)-- (9.,6.);
\draw (9.,6.)-- (9.,1.);
\draw (9.,1.)-- (1.,1.);
\draw (1.,6.)-- (1.,1.);
\draw [rotate around={0.:(5.,3.3)}] (5.,3.3) ellipse (5.4cm and 3.1cm);
\draw [rotate around={0.:(3.5,3.3)}] (3.5,3.3) ellipse (1.8cm and 1.3cm);
\draw [rotate around={0.:(6.5,3.3)}] (6.5,3.3) ellipse (1.8cm and 1.3cm);
\draw (5,6.2) node[anchor=center] {Isomorphism-Invariant Random Digraphs};
\draw (5.,5.57) node[anchor=center] {VARD};
\draw (3.5,4.3) node[anchor=center] {ARD};
\draw (6.5,4.3) node[anchor=center] {VRD};
\end{tikzpicture}}
\caption{Venn diagram of vertex-arc random digraphs for $n\geq 4$.
The results of this paper show that $\text{ARD}\cap \text{VRD} = \emptyset$ and all the four regions in the figure are nonempty for $n\geq 4$.
}
\label{fig:vards}
\end{figure}
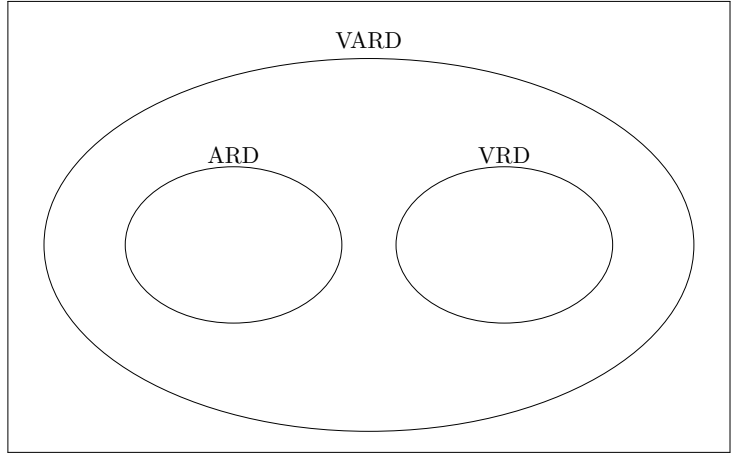

\begin{remark}
{\bf Approximation to VARDs by VRDs.}
However, any VARD can be arbitrarily closely approximated by VRDs.
That is, for any VARD $\mathbf{D}$ and $\epsilon>0$,
there exists a VRD $\mathbf{D}'$ such that $d_{ \text{TV} } (\mathbf{D},\mathbf{D}') < \epsilon$.
This result is a straightforward extension of approximation of VERGs by VRGs which immediately follows by letting
$\psi(y_1,y_2)$ to be the indicator of the event $\phi(x_1,x_2)\geq f_1(a_2)$ in the proof of Theorem 3.3 in \cite{beer:2011}.
\end{remark}

\section{Direction random digraphs}
\label{sec:drd}
One can also obtain isomorphism-invariant random digraphs by first generating an isomorphism-invariant random graph
and then assigning directions randomly to each edge.
Along this line, we first generate an isomorphism-invariant random graph, $\mathbf{G}= ( \mathcal{G}_n, P_{\mathbf{G}})$,
and then for each edge $ij\in E(\mathbf{G})$, independent of other edges,
pick a one sided or two sided direction randomly between $i$ and $j$.
For a given direction probability $1/2 \leq p_d <1$,
we put only the arc $(i,j)$ with probability $1-p_d$, only the arc $(j,i)$ with probability $1-p_d$ and
both of the arcs with probability $2p_d-1$.
Observe that the arc $(i,j)$ is put with probability $p_d$.
Also, note that we omit the case $p_d=1$ because it removes randomness in the direction.

The \emph{underlying graph} of a digraph $D$, denoted $U(D)$, is the graph obtained by replacing each arc of $D$ with an edge, disallowing multiple edges between two vertices (\cite{chartrand:1996}).
\begin{definition}
The \emph{underlying random graph} of a random digraph $\mathbf{D}= ( \mathcal{D}_n, P_{\mathbf{D}})$ is the random graph
$\mathbf{G}= ( \mathcal{G}_n, P_{\mathbf{G}})$ such that
\begin{align*}
P_{\mathbf{G}}(G)=\sum_{U(D)=G} P_{\mathbf{D}}(D) \text{ for every } G\in \mathcal{G}_n.
\end{align*}
\end{definition}

For instance, the underlying random graph of an ARD, $\mathbf{D}(n,p_a)$, is an ERG, namely,
$\mathbf{G}(n,p_e)$ with $p_e=2p_a-p_a^2$.
Moreover, notice also that the underlying random graph of a VARD, $\mathbf{D}(n, \Omega,\mu, \phi)$, is the VERG,
$\mathbf{G}(n, \Omega,\mu, \phi_u)$, where $\phi_u(x,y)=\phi(x,y)+\phi(y,x)-\phi(x,y)\phi(y,x)$.
In particular, the underlying random graph of a VRD is a VRG.

For a digraph $D\in \mathcal{D}_n$,
let $n_a(D)=|A(D)|$ and $n_e(D)=|E(U(D))|$ (i.e., the number of edges of the underlying graph of $D$).
Also, let $n_s(D)$ denote the number of pairs of vertices $i$ and $j$ such that both $(i,j)$ and $(j,i)$ are in $A(D)$
(i.e., the number of symmetric arcs in $D$),
and $n_{as}(D)$ denote the number of arcs $(i,j)$ in $A(D)$ with $(j,i) \notin A(D)$.
We write $n_a,n_e,n_s$ and $n_{as}$, respectively, dropping the digraph $D$ in the notation for brevity.
Note that $n_e=n_s+n_{as}$ and $n_a=2n_s+n_{as}$.

\begin{definition}\label{def:drd}
Let $\mathbf{G}= ( \mathcal{G}_n, P_{\mathbf{G}})$ be an isomorphism-invariant random graph and $1/2\leq p_d < 1$.
A \emph{direction random digraph} (DRD) is a random digraph $\mathbf{D} =( \mathcal{D}_n, P)$ with
\begin{align*}
P(D)=P_{\mathbf{G}}(U(D)) (1-p_d)^{n_{as}} (2p_d-1)^{n_s}  \text{ for every }  D\in \mathcal{D}_n ,
\end{align*}
and we say that $\mathbf{D}$ is \emph{generated by} $\mathbf{G}$ with \emph{direction probability} $p_d$.
\end{definition}

A natural question is why not start with a non-random graph and insert directions randomly to the edges to obtain DRDs.
There is a simple answer to the question.
Unfortunately, if directions are randomly inserted to the edges of a (fixed) graph,
the resulting random digraph is not isomorphism-invariant unless we start with an \emph{empty graph} (the graph with no edges) or a \emph{complete graph} (the graph with all possible edges).
Notice that $ \mathbf{G}$ is the underlying random graph of a DRD generated by $\mathbf{G}$.
Observe that if the digraphs $D_1$ and $D_2$ are isomorphic, then so are the (underlying) graphs $U(D_1)$ and $U(D_2)$,
and we also have $ n_s(D_1)=n_s(D_2)$ and $n_{as}(D_1)=n_{as}(D_2)$.
Thus, a DRD is isomorphism-invariant only if it is generated by an isomorphism-invariant random graph.
moreover, notice that we may consider a (fixed) graph as a degenerate random graph.
Also, it is easy to see that the empty graph and the complete graph with vertex set $[n]$ are the only graphs in $\mathcal{G}_n$ which are isomorphic to no other graph in $\mathcal{G}_n$,
and therefore these two graphs are the only isomorphism-invariant degenerate random graphs.

\subsection{DERDs, DVRDs and DVERDs}
We next provide three classes of direction random digraphs which are generated by ERGs, VRGs or VERGs.

\begin{definition}\label{def:derd}
The direction random digraph generated by an ERG, $\mathbf{G}(n,p_e)$, with direction probability $p_d$ is called
\emph{direction-edge random digraph} (DERD) and denoted $\mathbf{D}(n,p_e, p_d)$.
\end{definition}

Notice that letting $p_d$ to be 1/2 avoids symmetric arcs, and hence in the case of $p_d=1/2$,
after generating an ERG each edge is independently oriented in one of the two directions with equal probability
(e.g., see the model in \cite{Subramanian:2003}).
For example, letting $p_e=1$ and $p_d=1/2$ gives a \emph{random tournament} in which
each edge of a complete graph is independently oriented in one direction with equal probability.
For more information about tournaments, see \cite{moon:1968}.

\begin{definition}\label{def:dvrd}
A direction random digraph generated by a VRG is called \emph{direction-vertex random digraph} (DVRD).
A \emph{direction-vertex-edge random digraph} (DVERD) is a direction random digraph generated by a VERG.
\end{definition}

Notice that the underlying random graphs of a DERD, a DVRD and a DVERD are an ERG, a VRG and a VERG, respectively.
Clearly any ERG or VRG is a VERG, and hence every DERD and DVRD has a DVERD representation.
In addition, the results in \cite{beer:2011} imply the following:
A non-degenerate DRD which is both a DERD and a DVRD is either with $n\leq 3$ or generated by an ERG, $\mathbf{G}(n,p_e)$, with $p_e=1$.
For every $n\geq 6$, there exist DVERDs which are neither DERDs nor DVRDs.
Moreover, for $n\geq 3$, there exist DRDs which are not among DVERDs, and
for $n\leq 3$, any DVERD is also a DVRD.
These results are illustrated in Figure \ref{fig:drds}.

\begin{remark}
{\bf Approximation to DVERDs by DVRDs.}
Let $\mathbf{D}= ( \mathcal{D}_n, P_{\mathbf{D}})$ be a DVERD generated by a VERG,
$\mathbf{G}= ( \mathcal{G}_n, P_{\mathbf{G}})$, with direction probability $p_d$.
By Theorem 3.3 in \cite{beer:2011}, for any $\epsilon >0$ there exists a VRG, $\mathbf{G}'=( \mathcal{G}_n, P_{\mathbf{G}'})$, satisfying $d_{ \text{TV}}(\mathbf{G}, \mathbf{G}') < \epsilon$.
Let $\mathbf{D}'= ( \mathcal{D}_n, P_{\mathbf{D}'})$ be the DVRD generated by $\mathbf{G}'$ with the same direction probability $p_d$.
Then, it is easy to see that
\begin{align*}
\sum_{U(D)=G} |P_{\mathbf{D}}(D) -P_{\mathbf{D}'}(D)|  = | P_{\mathbf{G}}(G)-P_{\mathbf{G}'}(G)|
\end{align*}
for every $G\in \mathcal{G}_n$,
and therefore we get $d_{ \text{TV}}(\mathbf{D}, \mathbf{D}')=d_{ \text{TV}}(\mathbf{G}, \mathbf{G}')$
which implies that the total deviation distance between $\mathbf{D}$ and $\mathbf{D}'$ is less than $\epsilon$.
\end{remark}

\begin{figure}\center
\scalebox{.8}{
\begin{tikzpicture}[line cap=round,line join=round,>=triangle 45,x=1.5cm,y=1.5cm]
\clip(0.8537834428690544,0.8212498658964582) rectangle (9.162062220187407,6.526727615465726);
\draw (1.,6.)-- (9.,6.);
\draw (9.,6.)-- (9.,1.);
\draw (9.,1.)-- (1.,1.);
\draw (1.,6.)-- (1.,1.);
\draw [rotate around={0.:(5.,3.3)}] (5.,3.3) ellipse (5.4cm and 3.1cm);
\draw [rotate around={0.:(4.3,3.3)}] (4.3,3.3) ellipse (1.5cm and 1.3cm);
\draw [rotate around={0.:(5.7,3.3)}] (5.7,3.3) ellipse (1.5cm and 1.3cm);
\draw [rotate around={0.:(5.0,3.3)}] (5.,3.3) ellipse (3.3cm and 2.3cm);
\draw (5,6.2) node[anchor=center] {Isomorphism-Invariant Random Digraphs};
\draw (5.,5.57) node[anchor=center] {DRD};
\draw (4.3,4.3) node[anchor=center] {DERD};
\draw (5.7,4.3) node[anchor=center] {DVRD};
\draw (5,5.) node[anchor=center] {DVERD};
\end{tikzpicture}}
\caption{Venn diagram of direction random digraphs.
The results of the paper imply that all the six regions in the figure are nonempty.
In particular, the region $\text{DERD} \cap \text{DVRD}$ only consists of DRDs with $n\leq 3$ and DRDs generated by $\mathbf{G}(n,p_e=1)$.
}
\label{fig:drds}
\end{figure}
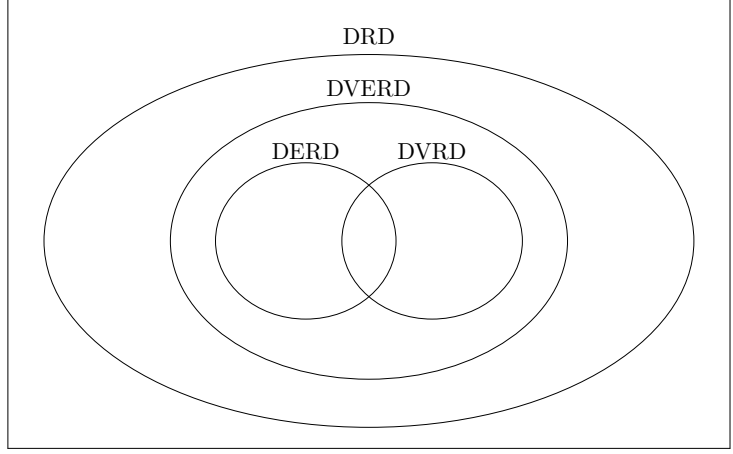

\section{Inclusion/exclusion relations of DERDs with respect to VARDs}
\label{sec:derdvard}
In this section, for $n\geq 4$, we show that a random digraph is both a DERD and a VARD if and only if it is an ARD, and
any DERD with $n\leq 3$ is also a VARD.
\begin{prop}
\label{prop:derdandard}
A DERD, $\mathbf{D}(n,p_e,p_d)$, is an ARD, $\mathbf{D}(n,p_a)$, if and only if
\begin{align*}
p_d=\frac{1}{1+\sqrt{1-p_e}} \text{  and  } p_a=1- \sqrt{1-p_e}.
\end{align*}
\end{prop}
\begin{proof}
Suppose that $\mathbf{D}(n,p_e,p_d)$ is an ARD $\mathbf{D}(n,p_a)$.
Then we have $p_ep_d=p_a$ since both are $P((1,2)\in A(\mathbf{D}))$.
Similarly we have $p_e(2p_d-1)=p_a^2$ as both are $P( \{ (1,2), (2,1) \} \subset A(\mathbf{D}))$.
Solving these two equations gives $p_d=(1\pm \sqrt{1-p_e})/p_e$.
If $p_e=1$, then definitely $p_d=1$.
Otherwise, $(1- \sqrt{1-p_e})/p_e <1<(1+ \sqrt{1-p_e})/p_e$ and hence $p_d=(1- \sqrt{1-p_e})/p_e$.
Note that $(1- \sqrt{1-p_e})/p_e= 1/(1+ \sqrt{1-p_e})$ and so $1/2\leq p_d \leq 1$.
Finally, $p_a=p_e p_d=1- \sqrt{1-p_e}$.

We next show that whenever $p_d=1/(1+\sqrt{1-p_e})$ and $p_a=1- \sqrt{1-p_e}$,
$\mathbf{D}(n,p_e,p_d)$ is $\mathbf{D}(n,p_a)$.
Note that in that case, $1-p_d=p_a(1-p_a)/p_e$, $2p_d-1=p_a^2/p_e$ and $1-p_e=(1-p_a)^2$.
Therefore, for a given $D\in \mathcal{D}_n$ we have
\begin{align*}
P(D)&=p_e^{n_e} (1-p_e)^{{n\choose 2}-n_e} (1-p_d)^{n_{as}} (2p_d-1)^{n_s}\\
&=p_e^{n_e} (1-p_e)^{(n(n-1)-2n_e)} \frac{p_a^{n_{as}}(1-p_a)^{n_{as}}}{p_e^{n_{as}}} \frac{p_a^{2n_s}}{p_e^{n_s}}\\
&=p_e^{n_e-n_{as}-n_s}  p_a^{n_{as}+2n_s} (1-p_a)^{n(n-1)-2n_e+n_{as}}\\
&=p_a^{n_a} (1-p_a)^{n(n-1)-n_a},
\end{align*}
since $n_e=n_{as}+n_s$ and $n_a=n_{as}+2n_s$.
Thus, the desired result follows.
\end{proof}

In fact, for $n\geq4$, the family of ARDs is the intersection of the classes DERDs and VARDs.
\begin{thm}
\label{thm:derdandvard}
If a DERD, $\mathbf{D}(n,p_e,p_d)$, with $n\geq 4$  has a VARD representation $\mathbf{D}(n, \Omega,\mu, \phi)$,
then $p_d = 1/(1+ \sqrt{1-p_e})$ and $\phi(x,y)=p_ep_d\ $ $\mu^2$-a.s.
\end{thm}
\begin{proof}
Suppose $\mathbf{D}(n,p_e,p_d)$ has a VARD representation $\mathbf{D}(n, \Omega,\mu, \phi)$.
Note that, as in any ARD, the events $(i,j)\in A(\mathbf{D})$ and $(k,l)\in A(\mathbf{D})$ are independent in a DERD whenever $\{i,j\}\neq \{k,l\}$.
Therefore, one can apply the method used in the proof of Theorem \ref{thm:1} and obtain
\begin{align}
\label{eq:derd1}
\phi(x,y) \phi(y,x)=p_e(2p_d-1) ~~~ \mu^2 \text{-a.s.}
\end{align}
and
\begin{align}
\label{eq:derd2}
(1-\phi(x,y))(1-\phi(y,x))=1-p_e ~~~ \mu^2 \text{-a.s.}
\end{align}
Solving the equations in \eqref{eq:derd1} and \eqref{eq:derd2} yields
\begin{align}
\phi(x,y)+\phi(y,x)=2p_ep_d ~~~ \mu^2 \text{-a.s.} \label{eq:derdsum}
\end{align}
and
\begin{align}
\phi (x,y)=p_ep_d \pm \sqrt{(1-\sqrt{1-p_e}-p_e p_d) (1+\sqrt{1-p_e}-p_e p_d)}~~~ \mu^2 \text{-a.s.} \label{eq:derd3}
\end{align}
If $p_d > 1/(1+\sqrt{1-p_e})=(1-\sqrt{1-p_e})/p_e$,
then the numbers in the right-hand side of \eqref{eq:derd3} have imaginary parts,
and hence we get a contradiction since $\phi$ takes only real values.

Suppose $p_d \leq 1/(1+\sqrt{1-p_e})$.
In any VARD, $\mathbf{D}(n,\Omega,\mu,\phi)$, with $n\geq 3$ we have
\begin{align}
P(\{ (1,2),(1,3)\}\subset A(\mathbf{D}))&=\int \int \int \phi(x_1,x_2)\phi(x_1,x_3)d(\mu x_1)d(\mu x_2)d(\mu x_3)  \nonumber \\
&=\int \left( \int \phi(x_1,x_2)d(\mu x_2) \right ) \left( \int \phi(x_1,x_3)d(\mu x_3) \right )d(\mu x_1) \nonumber \\
&=\int \left( \int \phi(x_1,x_2)d(\mu x_2) \right )^2 d(\mu x_1) \nonumber \\
&\geq \left ( \int \int \phi(x_1,x_2)d(\mu x_1)d(\mu x_2) \right )^2=\left (P((1,2)\in A(\mathbf{D})) \right )^2 \label{eq:derd4}
\end{align}
by  Fubini's theorem and
the Cauchy-Schwarz inequality applied to the constant function $\mathbf{1}$ and $\int \phi(x_1,x_2)d(\mu x_2)$.
On the other hand, in any DERD, $\mathbf{D}(n,p_e,p_d)$, with $n\geq 3$ we have
\begin{align*}
P( \{ (1,2),(1,3)\} \subset A(\mathbf{D}))=(p_e p_d)^2=\left (P((1,2)\in A(\mathbf{D})) \right )^2.
\end{align*}
Therefore, we have the equality in the Cauchy-Schwarz inequality in \eqref{eq:derd4}.
Thus, $\int \phi(x_1,x_2)d(\mu x_2)=c$ $\mu$-a.s. for some constant $c$.
Since
\begin{align*}
p_ep_d=P((1,2)\in A(\mathbf{D}))=\int \int \phi(x_1,x_2)d(\mu x_2)d(\mu x_1)=\int c\ d(\mu x_1)=c,
\end{align*}
we obtain $c=p_ep_d$, that is,
\begin{align}
\int \phi(x,y)d(\mu y)=p_ep_d \ \ \mu \text{-a.s.} \label{eq:derdint1}
\end{align}

Similarly, in a VARD, $\mathbf{D}(n,\Omega,\mu,\phi)$, with $n\geq 4$ we have
\begin{align}
&P( \{ (1,2),(2,3), (1,4), (4,3)\} \subset A(\mathbf{D})) \nonumber \\
&=\int \int \int \int \phi(x_1,x_2)\phi(x_2,x_3) \phi(x_1,x_4)\phi(x_4,x_3)  d(\mu x_1)d(\mu x_2)d(\mu x_3) d(\mu x_4) \nonumber \\
&=\int \int \left( \int \phi(x_1,x_2) \phi(x_2,x_3)d(\mu x_2) \right ) \left( \int \phi(x_1,x_4) \phi(x_4,x_3)d(\mu x_4) \right )d(\mu x_1) d(\mu x_3) \nonumber \\
&=\int \int  \left( \int \phi(x_1,x_2)\phi(x_2,x_3)d(\mu x_2) \right )^2 d(\mu x_1) d(\mu x_3) \nonumber \\
&\geq \left ( \int \int \int \phi(x_1,x_2) \phi(x_2,x_3) d(\mu x_1)d(\mu x_2)d(\mu x_3) \right )^2
=\left (P( \{ (1,2),(2,3)\} \subset A(\mathbf{D})) \right )^2  \label{eq:derd5}
\end{align}
by  Fubini's theorem and
the Cauchy-Schwarz inequality applied to the constant function $\mathbf{1}$ and $\int \phi(x_1,x_2)\phi(x_2,x_3)d(\mu x_2) $. Since in a DERD, $\mathbf{D}(n,p_e,p_d)$, with $n\geq 4$ we have
\begin{align*}
P( \{ (1,2),(2,3), (1,4), (4,3)\} \subset A(\mathbf{D}))=(p_e p_d)^4=\left (P( \{(1,2), (2,3)\} \subset A(\mathbf{D})) \right )^2,
\end{align*}
we obtain $\int \phi(x_1,x_2)\phi(x_2,x_3)d(\mu x_2)$ is constant $ \mu^2$-a.s.
by the equality in Cauchy-Schwarz inequality in \eqref{eq:derd5}.
By the equality in \eqref{eq:derd5}, one can easily verify that
\begin{align}
\int \phi(x,y)\phi(y,z)d(\mu y)=(p_ep_d)^2\ \ \mu^2\text{-a.s.} \label{eq:derdint2}
\end{align}

Let $s(x,y)=i(\phi (x,y)-p_ep_d)$.
Combining the results in \eqref{eq:derdsum}, \eqref{eq:derd3}, \eqref{eq:derdint1} and \eqref{eq:derdint2} gives
\begin{align}
s(x,y)=i(\phi (x,y)-p_ep_d)=-i(\phi (y,x)-p_ep_d)=\overline{s(y,x)} \ \ \mu^2\text{-a.s.} \label{eq:derds1}
\end{align}
and
\begin{align}
\int s(x,y)s(y,z)d(\mu y)=0\ \ \mu^2\text{-a.s.} \label{eq:derds2}
\end{align}
Let $T$ be the integral operator with kernel $s$ on the space $L^2(\Omega,\mu)$
\begin{align*}
(Tg)(x)=\int s(x,y)g(y)d(\mu y).
\end{align*}
Since $s$ is bounded and $\mu$ is a finite measure, the kernel $s$ is in $L^2(\mu \times \mu )$.
Moreover, the integral operator $T$ is compact and self-adjoint by \eqref{eq:derds1},
which implies that $L^2(\Omega,\mu)$ has an orthonormal basis $(\psi_i)_{i\geq 1}$ of eigenfunctions for $T$
such that $T\psi _i=\lambda_i \psi_i$ for not necessarily distinct eigenvalues $\lambda_i$, and
\begin{align}
s(x,y)=\sum_{i\geq 1} \lambda_i \psi_i(x)\overline{\psi_i(y)} ~~~ \mu^2\text{-a.s.} \label{eq:derds3}
\end{align}
with the sum converging in $L^2$ (\cite{reedsimon:1980}).
Since $\psi_i$'s are orthonormal, the equations \eqref{eq:derds2} and \eqref{eq:derds3} imply
\begin{align}
\sum_{i\geq 1} \lambda_i^2 \psi_i(x)\overline{\psi_i(z)}=0 ~~~ \mu^2\text{-a.s.} \label{eq:derds4}
\end{align}
Therefore, for any $m\geq 1$, by multiplying the equation in \eqref{eq:derds4} by $\psi_m(z)\overline{\psi_m(x)}$ and integrating over $x$ and $z$ we obtain $\lambda_m^2=0$, i.e., $\lambda_m=0$ for each $m$.
Thus, $s(x,y)=0\ \mu^2$-a.s. and hence $\phi(x,y)=p_ep_d \ \mu^2$-a.s. which implies $p_d = 1/(1+\sqrt{1-p_e})$.
\end{proof}

\begin{remark}
Notice that Theorem \ref{thm:derdandvard} and Proposition \ref{prop:derdandard} together imply Theorem \ref{thm:1}.
However, we provide the proof of Theorem \ref{thm:1} to keep the proof of Theorem \ref{thm:derdandvard} shorter and also to point out similarities and the differences with the techniques used in \cite{beer:2011}.
\end{remark}

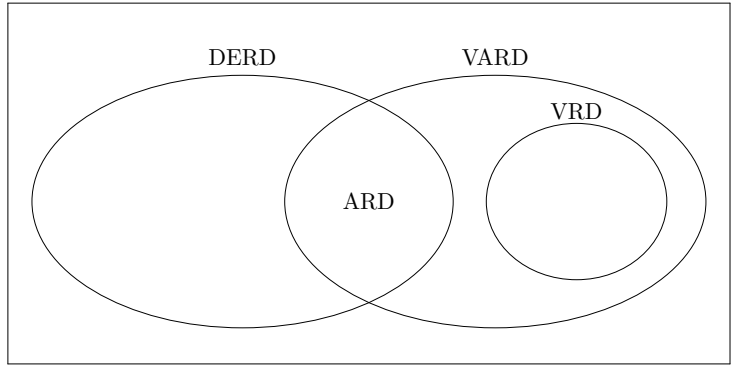
\begin{figure}\center
\scalebox{.8}{
\begin{tikzpicture}[line cap=round,line join=round,>=triangle 45,x=1.5cm,y=1.5cm]
\clip(0.8537834428690544,2.4212498658964582) rectangle (9.162062220187407,7.7);
\draw (1.,6.5)-- (9.,6.5);
\draw (9.,6.5)-- (9.,2.5);
\draw (9.,2.5)-- (1.,2.5);
\draw (1.,6.5)-- (1.,2.5);
\draw [rotate around={0.:(3.6,4.3)}] (3.6,4.3) ellipse (3.5cm and 2.1cm);
\draw [rotate around={0.:(6.4,4.3)}] (6.4,4.3) ellipse (3.5cm and 2.1cm);
\draw [rotate around={0.:(7.3,4.3)}] (7.3,4.3) ellipse (1.5cm and 1.3cm);
\draw (5,6.7) node[anchor=center] {Isomorphism-Invariant Random Digraphs};
\draw (5.,4.3) node[anchor=center] {ARD};
\draw (3.6,5.9) node[anchor=center] {DERD};
\draw (6.4,5.9) node[anchor=center] {VARD};
\draw (7.3,5.3) node[anchor=center] {VRD};
\end{tikzpicture}}
\caption{Venn diagram of DERDs and VARDs for $n\geq 4$.
The results in this paper indicate that all the five regions in the figure are nonempty.
In addition, the intersection of the classes DERDs and VARDs is the family of ARDs,
i.e., $\text{DERD}\cap \text{VARD}= \text{ARD}$.
}
\label{fig:derdvard}
\end{figure}

However, for $n\leq 3$, any DERD has a VARD representation.
\begin{thm}
\label{thm:derd3isvard}
Any DERD, $\mathbf{D}(n,p_e,p_d)$, with $n\leq 3$ is also a VARD.
\end{thm}
\begin{proof}
Let $\oplus$ and $\ominus$ denote addition and subtraction modulo 1, respectively.
In other words, for real numbers $0\leq x,y <1$,
\begin{align*}
x\oplus y =
  \begin{cases}
   x+y, & \text{if } x+y<1 \\
   x+y-1, &  \text{if } x+y \geq 1
  \end{cases}
\end{align*}
and
\begin{align*}
x\ominus y =
  \begin{cases}
   x-y, & \text{if } x-y\geq 0 \\
   x-y+1, & \text{if } x-y< 0
  \end{cases}
\end{align*}
If $U_1,U_2,U_3$ are independent uniform random variables over $[0,1)$,
then so are $U_1\oplus U_2, U_2\oplus U_3$ and $U_3\oplus U_1$ (see Lemma 4.5 in \cite{beer:2011}).
Therefore, $\mathbf{G}(3, p_e)$ can be represented as a vertex random graph $\mathbf{G}(3,[0,1), \nu,f)$ where
$\nu$ is the uniform distribution on $[0,1)$ and $f(x,y)=\mathbf{1}_{ \{ x\oplus y \leq p_e \}}$ (\cite{beer:2011}).

Let $g(x,y)=\mathbf{1}_{ \{x\ominus y\leq 1/2 \}}+(2p_d-1)\mathbf{1}_{ \{ x\ominus y>1/2 \} }$ for every $0\leq x,y <1$.
We claim that $\mathbf{D}(3,p_e,p_d)$ is a VARD, $\mathbf{D}(3,\Omega ,\mu, \phi)$, where
$\Omega=[0,1)\times [0,1)$, $\mu$ is product of two uniform distributions on $[0,1)$ and
$\phi( (u_1,u_1'), (u_2,u_2'))=f(u_1,u_2)g(u_1',u_2')$.
First note that
\begin{align}
g(x,y)+g(y,x)=2p_d \ \text{ and } \ g(x,y)g(y,x)=2p_d-1, \label{eq:derd3n1}
\end{align}
for every $ 0\leq x,y<1$.
As $f$ is a symmetric indicator function, the equations in \eqref{eq:derd3n1} imply
\begin{align}
\phi( (u_1,u_1'), (u_2,u_2')) \phi( (u_2,u_2'), (u_1,u_1')) &=f(u_1,u_2)(2p_d-1), \label{eq:derd3n2}\\
\phi( (u_1,u_1'), (u_2,u_2')) (1- \phi( (u_2,u_2'), (u_1,u_1'))) &=f(u_1,u_2)( g(u_1',u_2')-(2p_d-1)), \label{eq:derd3n3} \\
(1-\phi( (u_1,u_1'), (u_2,u_2'))) (1- \phi( (u_2,u_2'), (u_1,u_1'))) &=1-f(u_1,u_2). \label{eq:derd3n4}
\end{align}

We next focus on the function $g$.
It is easy to see that
\begin{align}
\int_0^1 g(x,y)~ dy= \int_0^1 g(x,y)~ dx=\frac{1}{2} 1+\frac{1}{2} (2p_d-1)=p_d, \label{eq:derd3n5}
\end{align}
for every $0\leq x,y<1$.

Consider the circle obtained by identifying the end points of the interval $[0,1]$ such that
$1/4$ is on the arc that starts from 0 and ends at $1/2$ along the clockwise direction.
Then, $x\ominus y$ is equal to the length of the arc of this circle
which starts from $x$ and ends at $y$ along the counterclockwise direction.
Notice that $g(x,y)g(y,z)$ and $g(x,y)g(y,z)g(z,x)$ depends on the ordering of $x,y,z$ along counterclockwise direction and
whether the points $x,y,z$ form an acute or obtuse triangle.

If $x,y,z$ form an acute triangle, there are basically two cases for the ordering, $x,y,z$ or $x,z,y$.
In the first case, $g(x,y)g(y,z)=(2p_d-1)^2$, and in the latter case $g(x,y)g(y,z)=1^2=1$.

If $x,y,z$ form an obtuse triangle, all six permutations of $x,y,z$ (x,y,z; x,z,y; y,x,z; y,z,x; z,x,y; z,y,x) are possible with the point at the middle corresponding to the obtuse angle.
Then, we have $g(x,y)g(y,z)= (2p_d-1)^2, (2p_d-1), (2p_d-1), (2p_d-1), (2p_d-1), 1$, respectively.
Moreover, it easy to show that three uniformly at random points on the circle form an acute triangle with probability $1/4$.
Therefore, we obtain
\begin{align}
\int_{[0,1)^3} g(x,y)g(y,z) ~ dx dy dz =\frac{1}{4}\cdot \frac{1}{2}((2p_d-1)^2+1)+
\frac{3}{4} \cdot \frac{1}{6} ((2p_d-1)^2+4(2p_d-1)+1) =p_d^2. \label{eq:derd3n6}
\end{align}
Similarly, we have
\begin{align}
\int_{[0,1)^3} g(x,y)g(y,z)g(z,x) ~ dx dy dz =\frac{1}{8}((2p_d-1)^3+1)+
\frac{1}{8} (3(2p_d-1)^2+3(2p_d-1)) =p_d^3. \label{eq:derd3n7}
\end{align}
By using the results in \eqref{eq:derd3n1}-\eqref{eq:derd3n7}, one can easily verify that
\begin{align*}
 \int P_{\mathbf{x}}(D) d(\mu \mathbf{x})=p_e^{n_e} (1-p_e)^{3-n_e} (1-p_d)^{n_{as}} (2p_d-1)^{n_s},
\end{align*}
for every $D\in \mathcal{D}_3$, and hence the desired result follows.
Furthermore, the same setting works for $n=2$ as well.
\end{proof}

\begin{remark}
{\bf Is every DERD with $n=3$ a VRD?}
Note that the function $\phi$ constructed in the proof of Theorem \ref{thm:derd3isvard} is binary (only takes the values 0 or 1) if and only if $p_d=1/2$.
Hence, by Theorem \ref{thm:derd3isvard} we see that any $\mathbf{D}(3,p_e,1/2)$ has a VRD representation.
But, by Proposition \ref{prop:derdandard}, $\mathbf{D}(3,p_e,1/2)$ is an ARD only if $p_e=0$ which gives a degenerate random digraph, and hence $\mathbf{D}(3,p_e,1/2)$ does not yield a non-degenerate ARD.
Other than the degenerate ones, is there any DERD $\mathbf{D}(3,p_e,p_d)$ with $p_d>1/2$ which is also a VRD?
Furthermore, is there an ARD with $n=3$ which has a VRD representation?
For now, these questions remain to be open, but, we conjecture that any DERD with $n=3$ is also a VRD.
\end{remark}

However, for $n=2$ the families DERDs and VRDs coincide with the all isomorphism-invariant random digraphs.
Let $D_1,D_2,D_3$ and $D_4$ be the digraphs with vertex set $[2]$ which only has the arc (1,2), only the arc (2,1), both of the arcs and none of the arcs, respectively.
Note that to obtain an isomorphism-invariant random digraph necessary and sufficient condition is $P(D_1)=P(D_2)$.
Let $\mathbf{D}$ be the random digraph with $P(D_1)=P(D_2)=p_1$, $P(D_3)=p_2$ and $P(D_4)=1-2p_1-p_2$.
First observe that $\mathbf{D}$ is an ARD if and only if $\sqrt{p_2}(1-\sqrt{p_2})=p_1$.
Letting $p_e=2p_1+p_2$ and $p_d=(p_1+p_2)/(2p_1+p_2)$ gives that $\mathbf{D}$ is a DERD $\mathbf{D}(2,p_e,p_d)$.
With the same $p_e$ and $p_d$,
let $\phi( (u_1,u_1'), (u_2,u_2'))=\mathbf{1}_{ \{u_1\oplus u_1' \leq p_e\} } \mathbf{1}_{ \{ u_2\ominus u_2' \leq p_d \} }$.
Then, it is easy to see that $\mathbf{D}$ is also VRD $\mathbf{D}(2,\Omega ,\mu, \phi)$, where
$\Omega=[0,1)\times [0,1)$, $\mu$ is the uniform distribution on $[0,1)^2$.
Therefore, when $n=2$, any isomorphism-invariant random digraph is both a DERD and a VRD (hence also a VARD).

\begin{remark}
{\bf Positive Dependence:}
Recall that by the inequality in \eqref{eq:derd4}, for any {\rm VARD} $\mathbf{D}$ we have the positive dependence
\begin{align}
P( \{ (1,2),(1,3)\} \subset A(\mathbf{D}))\geq P((1,2)\in A(\mathbf{D}))P((1,3)\in A(\mathbf{D}))= P((1,2)\in A(\mathbf{D}))^2. \label{eq:posdep}
\end{align}
Furthermore, the inequality in \eqref{eq:posdep} can be generalized by H\"{o}lder's inequality as follows
\begin{align}
P( \{ (1,2),\dots ,(1,m)\} \subset A(\mathbf{D}))\geq \prod_{i=2}^m P((1,i)\in A(\mathbf{D})) =
P((1,2)\in A(\mathbf{D}))^{m-1} \label{eq:posdepm}
\end{align}
for every {\rm VARD} $\mathbf{D}$ and $2\leq m \leq n$.
Similarly, we have the same inequality in \eqref{eq:posdepm} for any {\rm DVERD} as well and
note that equality holds for every {\rm DERD}.
However, there are random digraphs other than {\rm DERD}s satisfying equality in \eqref{eq:posdepm} for each $m$.
For example, consider the {\rm VRD}, $\mathbf{D}(n,[0,1),\mu,\phi)$,
where $\mu$ is the uniform distribution over $[0,1)$ and $\phi(x,y)=\mathbf{1}_{ \{x\ominus y\geq 3/8 \}} \mathbf{1}_{ \{y\ominus x\geq 3/8 \}}$.
Clearly, in this case, we have $P( \{ (1,2), \dots ,(1,m) \} \subset A(\mathbf{D}))=(1/4)^{m-1}$ and
$P((1,i)\in A(\mathbf{D}))=1/4$ for each $i$.
Also, it is easy to verify that $\mathbf{D}(n,[0,1),\mu,\phi)$ has no {\rm DERD} representation since
$P( \{ (1,2),(1,3), (2,3) \} \subset A(\mathbf{D}))=0$.
In the same manner, one can easily obtain similar results for random graphs.
In other words, for any {\rm VERG}, $\mathbf{G}(n,\Omega,\mu,\phi)$, and $2\leq m \leq n$, we have
\begin{align}
P( \{ \{1,2\}, \{1,3\}, \dots , \{1,m\} \}  \subset E(\mathbf{G}))\geq P( \{1,2\} \in E(\mathbf{G}))^{m-1}, \label{eq:vergposdepm}
\end{align}
and equality holds for every {\rm ERG}.
The underlying random graph of the {\rm VRD}, $\mathbf{D}(n,[0,1),\mu,\phi)$, described above is an example for
random graphs with no {\rm ERG} representation which attains equality in \eqref{eq:vergposdepm} for every $m$.
\end{remark}

\section{Where do random nearest neighbor digraphs reside?}
\label{sec:rnnd}
We determine the class relationship for one of the most commonly studied random digraphs, namely,
random nearest neighbor (NN) digraphs (e.g., see \cite{friedman:1983}, \cite{eppstein:1997},\cite{cuzick:1990} and \cite{penrose:2001}).
Let $n\geq 3, k \geq 1$ and $d\geq 1$ be integers with $k<n-1$.
Let $\mu$ be a probability distribution over $\mathbb{R}^d$
with density function $f$ that is assumed to be continuous almost everywhere with respect to Lebesgue measure.
Let $|\cdot|$ denote a fixed norm on $\mathbb{R}^d$ and $X=(X_1,\dots ,X_n)$ be i.i.d. vectors in $\mathbb{R}^d$ drawn from $\mu$.

For given $\mathbf{x}=(x_1,\dots ,x_n)$, the set of $k$ \emph{nearest neighbors} ($k$NNs) of $x_i$ is the closest $k$ points to $x_i$
among the points $\{x_1,\dots ,x_n\}\backslash \{x_i\}$ with respect to the given norm $|\cdot |$
and denoted as $k\text{NN}_{\mathbf{x}}(x_i)$.
As the occurrence of a tie is an event with zero probability for points from an a.e. continuous $f$,
we may assume that $k\text{NN}_{\mathbf{x}}(x_i)$ is well defined for each $i$ with probability 1.
The $k$ \emph{nearest neighbor digraph} of $\mathbf{x}$ is the digraph with vertex set $V=[n]$
and the arc set $A=\{(i,j): x_j \in k\text{NN}_{\mathbf{x}}(x_i) \}$,
(i.e., the arc $(i,j)$ is inserted if and only if
$x_j$ is one of the $k$NNs of $x_i$) and denoted as $k\text{NND}(\mathbf{x})$.
\begin{definition}\label{def:rnnd}
The \emph{random nearest neighbor digraph} (RNND) is the random digraph $\mathbf{D}(n,[k],d, \mu, |\cdot |)$ with
\begin{align*}
P(D)=\int \mathbf{1}_{ \{ k\text{NND}(\mathbf{x})=D \}} d(\mu  \mathbf{x} ) \text{ for every } D\in \mathcal{D}_n.
\end{align*}
\end{definition}
Notice that we picked $k$ to be less than $n-1$, because otherwise, we obtain a degenerate random digraph.

\begin{prop}
Every RNND is isomorphism-invariant.
\end{prop}
\begin{proof}
Let $\mathbf{D}(n,[k],d,\mu, |\cdot |)$ be a RNND and $D,D'\in \mathcal{D}_n$ be isomorphic digraphs.
Then there exists a permutation $\sigma$  on $[n]$ such that
\begin{align*}
(i,j)\in A(D) \ \Leftrightarrow \ (\sigma(i),\sigma(j))\in A(D').
\end{align*}
Let $\sigma^{-1}$ be the inverse of $\sigma$ and $\mathbf{y}=(y_1,\dots ,y_n)$
such that $y_i=x_{\sigma^{-1}(i)}$ for all $1\leq i \leq n$, i.e., $y_{\sigma(i)}=x_i$.
Then it is easy to see that
\begin{align*}
k\text{NND}(\mathbf{x})=D \ \Leftrightarrow \ k\text{NND}(\mathbf{y})=D'.
\end{align*}
The rest of the proof is similar to that of Proposition \ref{prop:isoinv}.
\end{proof}

As in VRDs, in the construction of a RNND, once $\mathbf{x}$ is fixed, then the arcs are uniquely determined.
However, in a VRD, by definition, inserting the arc $(i,j)$ only depends on $x_i$ and $x_j$
whereas in a RNND it depends on all the data points.
The following proposition implies that a RNND is not a VRD.

\begin{prop}
\label{prop:onlyrnnd}
A RNND is neither a VARD nor a DRD.
\end{prop}
\begin{proof}
We first show that no RNND has a VARD representation.
Recall that by the inequality in \eqref{eq:posdep},
for any VARD with $n\geq 3$ we have
\begin{align}
P( \{ (1,2),(1,3)\} \subset A(\mathbf{D})) \geq \left (P((1,2)\in A(\mathbf{D})) \right )^2. \label{eq:rnotvd1}
\end{align}
On the other hand, in a RNND, we have
\begin{align}
P( \{ (1,2),(1,3) \} \subset A(\mathbf{D}))=\frac{k(k-1)}{(n-1)(n-2)}< \left(\frac{k}{n-1}\right)^2=(P((1,2)\in A(\mathbf{D})))^2  \label{eq:rnotvd2}
\end{align}
by symmetry, and hence the result follows by \eqref{eq:rnotvd1} and \eqref{eq:rnotvd2}.

We show that there is no RNND which is also a DRD by contradiction.
Suppose that a RNND, $\mathbf{D}=(\mathcal{D}_n, P)$, is a DRD
generated by the random graph $\mathbf{G}=(\mathcal{G}_n, P_{\mathbf{G}})$ and with direction probability $p_d$.
Let $G$ be a graph in $\mathcal{G}_n$ with $P_{\mathbf{G}}(G)>0$, and $D$ be a digraph in $\mathcal{D}_n$ such that
$U(D)=G$, $n_s(D)=0$ and containing a vertex which is the tail of no arc.
In other words, $D$ is a digraph whose underlying graph is $G$, containing no symmetric arcs and
there exists a vertex $v$ in $V(D)$ such that $v$ is the head of every arc incident to $v$.
Then, as $\mathbf{D}$ is a DRD, we have
\begin{align}
P(D)= P_{\mathbf{G}}(G)(1-p_d)^{n_{as}}>0, \label{eq:rnndanddrd}
\end{align}
since $P_{\mathbf{G}}(G)>0$ and $p_d<1$.
On the other hand, for any given $\mathbf{x}=(x_1,\dots ,x_n)$,
every vertex is the tail of exactly $k$ arcs in $k\text{NND}(\mathbf{x})$.
Therefore, we obtain $P(D)=0$ which contradicts with \eqref{eq:rnndanddrd}.
\end{proof}

\begin{remark}\label{rem:underrnnd}
For any set of points in $\mathbb{R}^d$, the number of points sharing a common $k$NN is bounded above by a constant which is independent of the number of points in the set (see, \cite{yukich:1998}).
That is, there exists a number $c$ which only depends on $d,k$ and the norm $|\cdot |$ such that
in any $k$NND a vertex is the head of at most $c$ arcs.
Therefore, a vertex of the underlying graph of a $k$NND is incident to at most $c+k$ edges.
Hence, if $\mathbf{G}$ is the underlying random graph of a RNND with $n\geq c+k+2$,
then we have $P( \{ \{1,2\}, \{1,3\}, \dots , \{1,n\} \}  \subset E(\mathbf{G}))=0$ which implies that
$\mathbf{G}$ is not a VERG by the inequality in \eqref{eq:vergposdepm}.
\end{remark}

\begin{remark}
One can also generate NN type random digraphs other than RNNDs.
For instance, in the construction of a RNND insert the arc $(i,j)$ if and only if $x_j$ is the $k$-th NN of $x_i$
(i.e., insert only the one to its $k$-th NN instead of putting arcs from each point to its all $k$NNs).
We can generalize RNNDs to $\mathbf{D}(n,S_k,d,\mu, |\cdot |)$ where $S_k$ is a nonempty subset of $[k]$ and
we insert the arc $(i,j)$ if and only if $x_j$ is the $s$-th NN of $x_i$ for some $s\in S_k$.
Then the results for RNNDs in this section are also valid for any $\mathbf{D}(n,S_k,d,\mu, |\cdot |)$,
i.e., every $\mathbf{D}(n,S_k,d,\mu, |\cdot |)$ is isomorphism-invariant, has no VARD or DRD representation, and for large $n$,
has an underlying random graph which is not a VERG.
\end{remark}

Note that Proposition \ref{prop:onlyrnnd} implies that the regions $\text{VARD}^c$ in Figure \ref{fig:vards},
$\text{DRD}^c$ in Figure \ref{fig:drds} and $(\text{DERD} \cup \text{VARD})^c$ in Figure \ref{fig:derdvard} are nonempty.

For $n=3$, the only possible value of $k$ is 1.
In this case, the pair with the minimum distance are NNs of each other
and the NN of the remaining point is one of the points in this pair.
Thus, by symmetry we have $P(A(\mathbf{D})=\{(i,j),(j,i),(k,i)\})=1/6$ for every pairwise distinct $i,j,k \in \{1,2,3\}$,
and therefore any RNND $\mathbf{D}(3,1,d,\mu, |\cdot |)$ is a uniform distribution over six digraphs
independent of $d,\mu$ and $|\cdot|$.
Also, note that the underlying random graph of $\mathbf{D}(3,1,d,\mu, |\cdot |)$ is always $\mathbf{G}(3,2)$.

Observe that any RNND and $\mathbf{D}(n,nk)$ have the same number of arcs.
However, these two random digraphs are different.
Because, for the event $E=\{ \{(1,2), \dots ,(1,k+2) \} \subset A (\mathbf{D} )  \}$ we have
$P(E)=0$ in a RNND since each vertex is tail of exactly $k$ arcs,
whereas $P(E)>0$ in $\mathbf{D}(n,nk)$ since $nk \geq k+1$.
Recall that $n_e=n_a-n_s$,
and hence the number of edges in the underlying graph of a $k$NND is $nk$ minus the number of symmetric arcs.
It is easy to see that for $n>3$ there exist $k$NNDs with different number of symmetric arcs,
and therefore the underlying random graph of a RNND with $n>3$ is not a $\mathbf{G}(n,m)$.

\section{Discussion and Conclusions}
\label{sec:dis}
In this paper, we present four families, namely, ARDs, VRDs, VARDs and DRDs, of isomorphism-invariant random digraphs based on where randomness resides.
First three of these classes are extensions of the isomorphism-invariant random graph classes presented in \cite{beer:2011} to digraphs.
The family of DRDs is obtained by randomly assigning directions to edges of isomorphism-invariant random graphs,
and includes three families, DERDs, DVRDs and DVERDs.

The main results of this paper are illustrated in Figures \ref{fig:vards}-\ref{fig:derdvard}.
For $n\geq 4$, we show that there is no random digraph that is both an ARD and a VRD
(which is the digraph counterpart of the result in \cite{beer:2011},
that there is no non-degenerate random graph which is both an ERG and a VRG for $n\geq 4$).
\cite{beer:2011} also show that for every $n\geq 6$, there exist VERGs which neither belong to ERGs nor VRGs.
We reduce the lower bound for $n$ from 6 to 4 by using non-symmetric structure of the function $\phi$, and obtain the digraph counterpart of their result, i.e.,
there exist VARDs which have no ARD or VRD representation, for $n\geq 4$.
However, for DRDs we have the same lower bound 6 for $n$; that is,
for $n\geq 6$ there exist DVERDs which neither belong to DERDs or DVRDs.
We also show that for $n\geq 4$ ARDs are the only random digraphs which have both DERD and VARD representations.
The method we use for the latter result is not applicable for the intersection of the families DVRDs and VARDs,
since we lose the independence of the edges in a VRG.
Therefore, identifying all random digraphs with both DVRD and VARD representations is a challenging problem, and remains open.

For $n=3$, we show that any DERD has a VARD representation and any DERD whose edge probability is $1/2$ is also a VRD.
However, the question whether there is other DERDs with VRD representation is open,
and we conjecture that any DERD is a VRD as well for $n=3$.
Yet, when $n=3$, every DVERD is a DVRD.
However, in the case of $n=2$, any isomorphism-invariant random digraph has DERD, DVRD and VRD representations.

We also study RNNDs and determine where they lie in these classifications.
We show that no RNND has a DRD or a VARD representation,
and the underlying random graph of a RNND with large $n$ is not a VERG.

\section*{Acknowledgments}
EC was supported by the European Commission under the Marie Curie International Outgoing Fellowship Programme
via Project \# 329370 titled PRinHDD.


\begin{thebibliography}{}

\bibitem[Beer et~al., 2011]{beer:2011}
Beer, E., Fill, J.~A., Janson, S., and Scheinerman, E.~R. (2011).
\newblock On vertex, edge, and vertex-edge random graphs.
\newblock {\em The Electronic Journal of Combinatorics}, 18(1):P110, 20pp.

\bibitem[Bloznelis, 2010]{bloznelis:2010}
Bloznelis, M. (2010).
\newblock A random intersection digraph: Indegree and outdegree distributions.
\newblock {\em Discrete Mathematics}, 310(19):2560--2566.

\bibitem[Bloznelis et~al., 2012]{bloznelis:2012}
Bloznelis, M., Gotze, F., and Jaworski, J. (2012).
\newblock Birth of a strongly connected giant in an inhomogeneous random
  digraph.
\newblock {\em Journal of Applied Probability}, 49(3):601--611.

\bibitem[Ceyhan, 2011]{ceyhan:2011}
Ceyhan, E. (2011).
\newblock Spatial clustering tests based on the domination number of a new
  random digraph family.
\newblock {\em Communications in Statistics -Theory and Methods},
  40(8):1363--1395.

\bibitem[Chartrand and Lesniak, 1996]{chartrand:1996}
Chartrand, G. and Lesniak, L. (1996).
\newblock {\em Graphs \& Digraphs}.
\newblock Chapman \& Hall/CRC, Boca Raton, FL.

\bibitem[Cuzick and Edwards, 1990]{cuzick:1990}
Cuzick, J. and Edwards, R. (1990).
\newblock Spatial clustering for inhomogeneous populations.
\newblock {\em Journal of the Royal Statistical Society. Series B
  (Methodological)}, 52(1):73--104.

\bibitem[Eppstein et~al., 1997]{eppstein:1997}
Eppstein, D., Paterson, M.~S., and Yao, F.~F. (1997).
\newblock On nearest-neighbor graphs.
\newblock {\em Discrete \& Computational Geometry}, 17(3):263--282.

\bibitem[Erd\H{o}s and R\'{e}nyi, 1959]{erdos:1959}
Erd\H{o}s, P. and R\'{e}nyi, A. (1959).
\newblock On random graphs {I}.
\newblock {\em Publicationes Mathematicae Debrecen}, 6:290--297.

\bibitem[Friedman and Rafsky, 1983]{friedman:1983}
Friedman, J.~H. and Rafsky, L.~C. (1983).
\newblock Graph-theoretic measures of multivariate association and prediction.
\newblock {\em The Annals of Statistics}, 11(2):377--391.

\bibitem[Gilbert, 1959]{gilbert:1959}
Gilbert, E.~N. (1959).
\newblock Random graphs.
\newblock {\em The Annals of Mathematical Statistics}, 30:1141--1144.

\bibitem[Graham and Pike, 2008]{graham:2008}
Graham, A.~J. and Pike, D.~A. (2008).
\newblock A note on thresholds and connectivity in random directed graphs.
\newblock {\em Atlantic Electronic Journal of Mathematics}, 3(1):1--5.

\bibitem[Karp, 1990]{Karp:1990}
Karp, R.~M. (1990).
\newblock The transitive closure of a random digraph.
\newblock {\em Random Structures Algorithms}, 1(1):73--93.

\bibitem[Krivelevich et~al., 2013]{Krivelevich:2013}
Krivelevich, M., Lubetzky, E., and Sudakov, B. (2013).
\newblock Longest cycles in sparse random digraphs.
\newblock {\em Random Structures Algorithms}, 43(1):1--15.

\bibitem[Luczak, 1990]{luczak:1990}
Luczak, T. (1990).
\newblock The phase transition in the evolution of random digraphs.
\newblock {\em Journal of Graph Theory}, 14(2):217--223.

\bibitem[Luczak and Seierstad, 2009]{luczak:2009}
Luczak, T. and Seierstad, T.~G. (2009).
\newblock The critical behavior of random digraphs.
\newblock {\em Random Structures Algorithms}, 35(3):271--293.

\bibitem[Moon, 1968]{moon:1968}
Moon, J.~W. (1968).
\newblock {\em Topics on Tournaments}.
\newblock Holt, Reinhart and Winston, New York.

\bibitem[Penrose, 2003]{penrose:2003}
Penrose, M. (2003).
\newblock {\em Random Geometric Graphs}.
\newblock Oxford University Press, New York.

\bibitem[Penrose and Yukich, 2001]{penrose:2001}
Penrose, M.~D. and Yukich, J.~E. (2001).
\newblock Central limit theorems for some graphs in computational geometry.
\newblock {\em The Annals of Applied Probability}, 11(4):1005--1041.

\bibitem[Reed and Simon, 1980]{reedsimon:1980}
Reed, M. and Simon, B. (1980).
\newblock {\em Methods of Modern Mathematical Physics. I. Functional Analysis}.
\newblock Academic Press, Inc., 2nd edition, San Diego, CA.

\bibitem[Subramanian, 2003]{Subramanian:2003}
Subramanian, C.~R. (2003).
\newblock Finding induced acyclic subgraphs in random digraphs.
\newblock {\em The Electronic Journal of Combinatorics}, 10:\#R46, 6pp.

\bibitem[Yukich, 1998]{yukich:1998}
Yukich, J.~E. (1998).
\newblock {\em Probability Theory of Classical Euclidean Optimization Problems.
  Lecture Notes in Mathematics}.
\newblock Springer, Berlin.

\end{thebibliography}
\end{document}